\documentclass{amsart}
\usepackage{amsmath,amssymb,amsthm,url,comment}
\usepackage{graphicx}

\newtheorem{theorem}{Theorem}[section]
\newtheorem{lemma}[theorem]{Lemma} 
\newtheorem{claim}{Claim} 
\newtheorem{proposition}[theorem]{Proposition} 
\newtheorem{conjecture}{Conjecture} 
 
\newtheorem{corollary}[theorem]{Corollary}
\theoremstyle{definition}
\newtheorem{definition}[theorem]{Definition}

\newtheorem{remark}[theorem]{Remark}
\newtheorem*{remark*}{Remark}

\numberwithin{equation}{section}

\newcommand{\Z}{\mathbb{Z}}

\newcommand{\C}{\mathbb{C}}
\newcommand{\Q}{\mathbb{Q}}
\begin{document}

\title[Chirally cosmetic surgeries of positive 2-bridge knot]{A non-torus positive 2-bridge knot does not admit chirally cosmetic surgeries}

\author{Tetsuya Ito}
\address{Department of Mathematics, Kyoto University, Kyoto 606-8502, JAPAN}
\email{tetitoh@math.kyoto-u.ac.jp}


\begin{abstract}
We prove that a positive two-bridge knot other than the $(2,k)$ torus knot does not admit chirally cosmetic surgeries, a pair of Dehn surgeries along distinct slopes yielding orientation-reversingly homeomorphic 3-manifolds.
\end{abstract}

\maketitle
\section{Introduction}

Let $K$ be a non-trivial knot in $S^{3}$ and $E(K)=S^{3} \setminus N(K)$ be its exterior, where $N(K)$ is an open tubular neighborhood of $K$. 
For a slope $r$, a homotopy class of unoriented simple closed curve on $\partial E(K)$, we denote by $S^{3}_K(r)$ the Dehn surgery on $K$ along the slope $r$.
Namely, $S^{3}_K(r)$ is the oriented closed 3-manifold obtained by attaching the solid torus $D^{2}\times S^{1}$ to $E(K)$ along $\partial E(K)$ so that the slope $r$ bounds a disk in the attached solid torus. As usual we identify the set of slopes with $\Q\cup \{ \infty = \frac{1}{0}\}$. We express slopes as $r=\frac{p}{q}$ by coprime integers $p,q$, and we always assume that $p>0$.

A pair of Dehn surgeries along distinct slopes $r$ and $r'$ are called \emph{chirally cosmetic} if $S^{3}_K(r) \cong -S^{3}_K(r')$. Here $-M$ denotes the 3-manifold $M$ with opposite orientation, and $M \cong N$ means that $M$ and $N$ are orientation-preservingly homeomorphic.

A famous (purely) cosmetic surgery conjecture says that \emph{purely cosmetic surgeries}, a pair of Dehn surgeries satisfying $S^{3}_K(r) \cong S^{3}_K(r')$ with $r \neq r'$ do not exist whenever $K$ is non-trivial. Contrary, there are two families of chirally cosmetic surgeries of knots in $S^{3}$.
\begin{itemize}
\item[(A)] If $K$ is amphicheiral, $S^{3}_K(r) \cong -S^{3}_K(-r)$ for all $r \in \Q$. 
\item[(B)] If $K$ is the $(2,k)$-torus knot, $\displaystyle S^{3}_K\left(\frac{2k^{2}(2m+1)}{k(2m+1)+1} \right) \cong -S^{3}_K\left(\frac{2k^{2}(2m+1)}{k(2m+1)-1}\right)$ for all $m \in \Z$ \cite{ro,iis1}. 
\end{itemize}
In \cite{it-LMO} we posed the following natural question which we state as a conjecture for convenience.

\begin{conjecture}[Chirally cosmetic surgery conjecture]
\label{conj:main}
Chirally cosmetic surgeries of knots in $S^{3}$ are either (A) or (B).
\end{conjecture}
 This optimistic conjecture has been studied and confirmed for several cases, such as, genus one alternating knots \cite{iis1}, alternating odd pretzel knots \cite{va1,va}, and iterated torus knots and most cable knots \cite{it-cable}.
 Also, the results in \cite{os-rational,it-LMO} and \cite{iis2} give several supporting evidences for the conjecture. 
 
Recently we confirmed the conjecture for special alternating knots with sufficiently large number (more than $81$) of twist regions \cite{it-special}. This justifies our intuition that chirally cosmetic surgeries are rare.

The aim of this paper is to show the conjecture for positive 2-bridge knots.

\begin{theorem}
\label{theorem:main}
A positive 2-bridge knot other than the $(2,k)$-torus knot does not admit chirally cosmetic surgeries.
\end{theorem}
Our argument says that for positive 2-bridge knots, Theorem \ref{theorem:main-obst}, our constraint for chirally cosmetic surgeries completely excludes the existence of chirally cosmetic surgeries. 

It deserves to mention and emphasize that our proof does not require any computer computations. Although the actual computation is tedious, one can check all the computations by hand. This makes a sharp contrast with \cite{hltz} where they confirmed Conjecture \ref{conj:main} for positive 2-bridge knot up to 31 crossings by computer calculations.

Although positive 2-bridge knots are special alternating knots, the strategy and arguments in this paper and that of \cite{it-special} are quite different.

In \cite{it-special} we exploit the fact that the determinant of alternating knots grows exponentially with respect to the number of twist regions of its reduced and twist-reduced alternating diagram \cite{st-determinant}. 
In \cite[Theorem 1.9]{iis2} we showed that when $K$ admits a chirally cosmetic surgeries, then we have a certain (in)equality between the determinant and several finite type invariants. Since finite type invariants grow polynomially with respect to the number of crossings, by a careful comparison of the growth we showed that when the number of twist region is large then (in)equality cannot hold hence $K$ never admits chirally cosmetic surgeries.

On the other hand, for positive two bridge knots, we use Conway's form as its input. 
Conway's form is a positive but non-alternating diagram. Thus it is no longer true that the determinant grows exponentially with respect to the number of twist regions. For two-bridge knot case we use induction argument. We show that once \emph{the main obstruction} \eqref{eqn:main-obstruction}, one of the sufficient condition for non-existence of chirally cosmetic surgery (Corollary \ref{cor:main-obstruction}) derived from our constraint (Theorem \ref{theorem:main-obst}), is satisfied for some $K$, then in almost all cases the knots obtained by adding crossing to $K$ also satisfy the main obstruction. The cases where the main obstruction does not apply, we use a different obstruction that again comes from Theorem \ref{theorem:main-obst}.

The organization of the paper is as follows.

In Section \ref{section:constraint} we prove Theorem \ref{theorem:main-obst}, our main constraint for knots to admit chirally cosmetic surgeries.
Although this is a simple combination of known constraints appeared in past researches \cite{it-LMO,iis1,iis2,os-rational,va}, it has not been appeared in literature and is quite useful and strong, as our main theorem demonstrates.

In Section \ref{section:2-bridge}, after reviewing basic facts on two bridge knots, we discuss our main constraint for two bridge knot case. 
A main technical result is Lemma \ref{lemma:signature-torus-knot}. The lemma shows that the signature of the torus knot $T_{q,n}$ is approximately $\frac{qn}{2}$. In Corollary \ref{cor:main-obstruction} we extract our main obstruction \eqref{eqn:main-obstruction} for 2-bridge knots to admit chirally cosmetic surgeries. We also present Theorem \ref{theorem:main-variant}, a useful variant of our constraint.

The proof of our main theorem is given in Section \ref{section:proof}.

\section*{Acknowledgement}

The author is partially supported by JSPS KAKENHI Grant Numbers 19K03490, 21H04428, 	23K03110.
The author wishes to express his gratitude to Marc Kegel, Kazuhiro Ichihara and Toshio Saito for valuable and stimulating discussions and comments.

\section{Chirally cosmetic surgery constraint}
\label{section:constraint}

In this section, we discuss constraints of chirally cosmetic surgeries. We refer to \cite{iis2} for some other constraints which are not used in this paper. See also \cite{fps} for hyperbolic geometry method that gives an effective and practical way to check whether a given hyperbolic knot admits (chirally or purely) cosmetic surgeries, though it requires careful computer calculations.

Let $a_2(K)$ and $a_4(K)$ be the coefficient of $z^{2}$ and $z^{4}$ of the Conway polynomial $\nabla_K(z)$ of $K$. 
Let 
\[ v_3(K) = -\frac{1}{144}V'''_K(1) - \frac{1}{48}V_K''(1) \in \frac{1}{4}\Z \]
where $V_K(t)$ is the Jones polynomial of $K$. The invariant $4v_3$ is known as the primitive finite type invariant of degree three, normalized so that $4v_3(T_{2,3})=1$ where $T_{2,3}$ is the positive $(2,3)$-torus knot\footnote{A somewhat unusual normalization that $v_3(K) \in \frac{1}{4}\Z$ comes from the Kontsevich invariant point of view \cite{it-LMO}.}.

For $\omega \in \{z \in \mathbb{C} \: | \: |z|=1\}$, let $\sigma_{\omega}(K)$ the Levine-Tristram signature\footnote{We adopt the convention that $\sigma_{\omega}(K)\geq 0$ for positive knot $K$. This is opposite to the convention in \cite{iis1}.} at $\omega$.
For $p>1$, the \emph{total $p$-signature} of $K$ is defined by 
\[ \sigma(K,p)=\sum_{\omega^{p}=1} \sigma_{\omega}(K). \]

To state our constraint, we introduce the following quantity.

\begin{definition}[Signature density]
For a knot $K$, we define the \emph{lower signature density} $\phi(K)$ and the \emph{upper signature density} $\Phi(K)$ by
\[ \phi(K)= \inf\left\{ \frac{|\sigma(K,p)|}{g(K)p} \: \middle| \: p \in \Z_{\geq 2} \right\},\quad \Phi(K)= \sup\left\{ \frac{|\sigma(K,p)|}{g(K)p} \: \middle| \: p \in \Z_{\geq 2} \right\}\]
respectively.
For $N\geq 0$ we define the \emph{lower signature density modulo $N$} $\phi^{\geq N}(K)$ as the \emph{upper signature density modulo $N$} $\Phi^{\geq N}(K)$ by
\[ \phi^{\geq N}(K)= \inf\left\{ \frac{|\sigma(K,p)|}{g(K)p} \: \middle| \: p \in \Z_{\geq N} \right\}, \quad  \Phi^{\geq N}(K)= \sup\left\{ \frac{|\sigma(K,p)|}{g(K)p} \: \middle| \: p \in \Z_{\geq N} \right\}\]
\end{definition}

Before stating our constraint, we review known facts on chirally cosmetic surgeries. 
First of all, the non-vanishing of $v_3(K)$ leads to the following useful consequence.
\begin{corollary}\cite[Corollary 1.3 (ii)]{it-LMO}
\label{cor:non-zero}
If $v_3(K) \neq 0$, then $S^{3}_K(r) \neq - S^{3}_K(-r)$.
\end{corollary}
We remark that $v_3(K) \neq 0$ implies that $K$ is not amphicheiral, because $v_3(K)=v_3(\overline{K})$ where $\overline{K}$ is the mirror image.

A knot $K$ is an \emph{$L$-space knot} if $S^{3}_K(r)$ is an L-space for some $r \geq 0$.
\begin{theorem}\cite[Theorem 1.6]{os-rational}
\label{theorem:L-space}
If $S^{3}_K(r) \cong -S^{3}_K(r')$ ($r\neq r'$) and $rr' \geq 0$, then $K$ is an L-space knot and $S^{3}_K(r)$ is an L-space.
\end{theorem}
These results give restrictions of possible slopes of chirally cosmetic surgeries.

Now we are ready to state our constraint for chirally cosmetic surgeries.

\begin{theorem}
\label{theorem:main-obst}
Assume that $K$ satisfies the following properties (i)--(vi).
\begin{itemize}
\item[(i)] $a_2(K)>1$.
\item[(ii)] $4v_3(K)> 0$.
\item[(iii)] $\sigma(K,p)> 0$ for all $p \geq 11$.
\item[(iv)] $g(K)=\tau(K)$ where $\tau(K)$ is the Heegaard Floer tau invariant \cite{os-tau}.
\item[(v)] $K$ is Heegaard Floer homologically thin.
\item[(vi)] $K$ is not an L-space knot.
\end{itemize}
If $K$ has chirally cosmetic surgeries, then $K$ satisfies 
\[ \frac{16a_2(K)}{\Phi^{\geq 11}(K)g(K)} \leq \frac{2(7a_2(K)^2-a_2(K)-10a_4(K))}{4v_3(K)} = \det(K) + 6g(K)-5 \leq \frac{16a_2(K)}{\phi^{\geq 11}(K)g(K)} \]
\end{theorem}
\begin{proof}
Assume that $S^{3}_K(p\slash q) \cong -S^{3}_{K}(p\slash q')$ for $p>0$.

From the degree two part of the LMO invariant, if $q+q'\neq 0$ and $v_3(K)\neq 0$ then
\begin{equation}
\label{eqn:LMO2}
\frac{p}{q+q'} = \frac{7a_2(K)^{2}-a_2(K)-10a_4(K)}{2(4v_3(K))} 
\end{equation}
holds \cite[Corollary 1.3]{it-LMO}.

Similarly, from the Heegaard Floer homology, under the assumption that $K$ satisfies the properties (iv) and (v), if $q+q'\neq 0$ and $qq'<0$ then
\begin{equation}
\label{eqn:Heegaard} 
\frac{p}{q+q'} = \frac{1}{2} \left(\frac{1}{2}\det(K)+ 3g(K)-\frac{5}{2} \right)
\end{equation}
holds \cite{va}.

By Corollary \ref{cor:non-zero}, the assumption (ii) implies that  $q+q' \neq 0$.
Similarly by Theorem \ref{theorem:L-space}, the assumption (vi) implies that $qq' <0$. Thus the equalities \eqref{eqn:LMO2} and \eqref{eqn:Heegaard} hold. 

Finally, by the Casson-Walker and the total Casson-Gordon invariants, we have seen that 
\begin{equation}
\label{eqn:CG-CW-obst} 12(q+q')a_2(K) = 3\sigma(K,p)
\end{equation}
holds \cite[Theorem 2.1]{iis1}. 
Furthermore, by \cite[Corollary 2.2]{iis1}, if $p\leq 10$ and $S^{3}_K(p\slash q) \cong -S^{3}_{K}(p\slash q')$, $a_2(K) \in \{1,0,-1\}$ holds (See Remark \ref{remark:improvement} for related discussions and improvements). Thus by the assumption (i), $p\geq 11$. Hence by the assumption (iii) we get
\begin{equation}
\label{eqn:obst1} \frac{4a_2(K)}{\Phi^{\geq 11}(K)g(K)}\leq \frac{p}{q+q'} = \frac{4a_2(K)p}{\sigma(K,p)} \leq \frac{4a_2(K)}{\phi^{\geq 11}(K)g(K)} 
\end{equation}

The statement follows from \eqref{eqn:LMO2}, \eqref{eqn:Heegaard}, and \eqref{eqn:obst1}
\end{proof}

\begin{remark}
\label{remark:improvement}
Although we used $\phi^{\geq 11}(K)$ and $\Phi^{\geq 11}(K)$, this is not optimal. 
In the proof of Theorem \ref{theorem:main-obst} we used \cite[Corollary 2.2]{iis1} that is obtained by computing the Casson-Walker and the Casson-Gordon invariant constraint \cite[Theorem 2.1]{iis1} for $p \leq 10$ cases. Further computations allows us to get to get more restrictions for $p,q,q',a_2(K)$ and $\sigma(K,p)$.
\end{remark}

There are many knots that satisfies all the properties (i)--(vi).
A non-trivial knot $K$ is \emph{positive} if it admits a diagram that has only positive crossings. Throughout the paper we regard the unknot as a non-positive knot.

\begin{proposition}
A positive knot $K$ satisfies the properties (ii)--(iv) of Theorem \ref{theorem:main-obst}. Furthermore, a positive knot $K$ satisfies the property (i) of Theorem \ref{theorem:main-obst} unless $K$ is the $(2,3)$-torus knot.
\end{proposition}
\begin{proof}
The property (ii), (iii), and (iv) are proven in \cite{st}, \cite{cg,pr}, and  \cite{li}, respectively. The last assertion is proven in \cite{st}. 
\end{proof}

A non-trivial knot $K$ is \emph{alternating} it it admits an alternating diagram. It is known that alternating knot satisfies the property (v) \cite[Theorem 1.3]{os-alt}. Furthermore, if an alternating knot is an L-space knot, then it is the $(2,k)$-torus knot \cite[Theorem 1.5]{os-lens}.

A knot $K$ is \emph{special alternating} if (up to mirror image) it is positive and alternating. The properties of positive knots and alternating knots say that a special alternating knot other than the $(2,k)$-torus knot satisfies the properties (i)--(iv) of Theorem \ref{theorem:main-obst}.

\begin{corollary}
\label{cor:main-obst}
Assume that $K$ is a special alternating knot other than the $(2,k)$-torus knot.
If $K$ has chirally cosmetic surgeries, then it satisfies 
\[ \frac{16a_2(K)}{\Phi^{\geq 11}(K)g(K)} \leq \frac{2(7a_2(K)^2-a_2(K)-10a_4(K))}{4v_3(K)} = \det(K) + 6g(K)-5 \leq \frac{16a_2(K)}{\phi^{\geq 11}(K)g(K)} \]
\end{corollary}

\section{Positive two bridge knot and chirally cosmetic surgery constraints}
\label{section:2-bridge}

\subsection{Conway's normal form}
For $g>0$ and $b_1,\ldots,b_g,c_1,\ldots,c_g \in \Z\setminus\{0\}$, we denote by $C[2b_g,2c_g,\ldots,2b_1,2c_1]$ the diagram (or, a knot by abuse of notation) as depicted in Figure \ref{fig:2-bridge} (i).

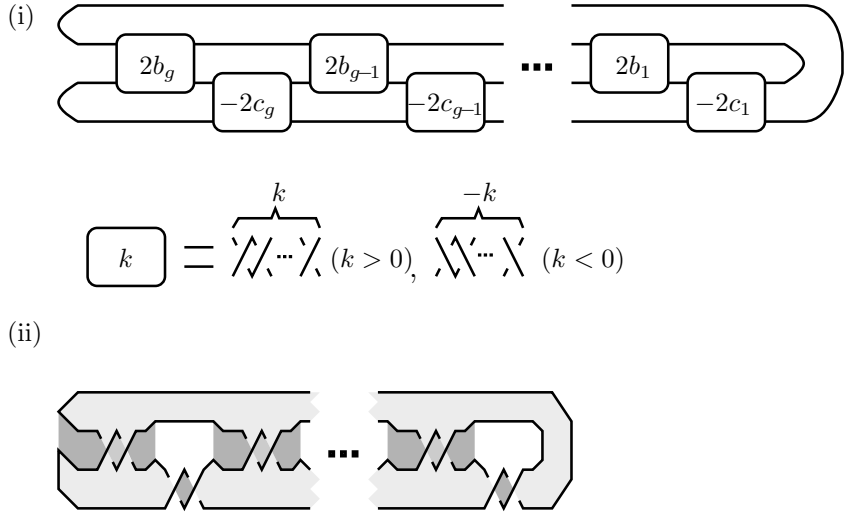
\begin{figure}[htbp]
\includegraphics*[width=105mm]{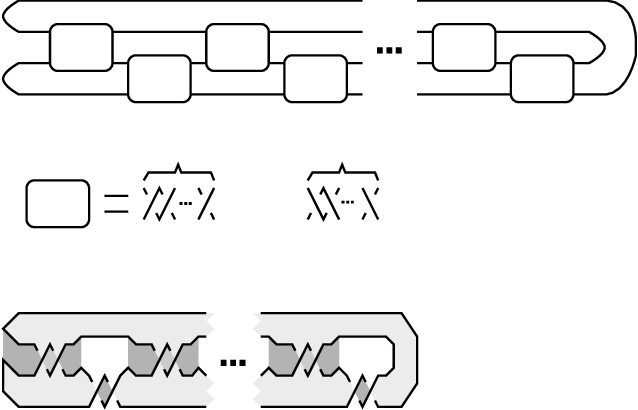}
\begin{picture}(0,0)
\put(-320,185){(i)}
\put(-320,65){(ii)}
\put(-270,166){$2b_g$}
\put(-240,152){$-2c_g$}
\put(-200,166){$2b_{g\!-\!1}$}
\put(-169,152){$-2c_{g\!-\!1}$}
\put(-90,166){$2b_1$}
\put(-60,152){$-2c_1$}
\put(-278,93){$k$}
\put(-198,93){($k>0$)}
\put(-168,89){\Large ,}
\put(-118,93){($k<0$)}
\put(-220,118){$k$}
\put(-148,118){$-k$}
\end{picture}
\caption{(i) Diagram $C[2b_g,2c_g,\ldots,2b_1,2c_1]$. (ii) The Seifert surface $S$} 
\label{fig:2-bridge}
\end{figure} 

It is known that a two-bridge knot $K$ of genus $g$ is represented by the diagram $C[2b_g,2c_g,\ldots,2b_1,2c_1]$ (see \cite[12.A]{bz}, for example).
We often exploit the following symmetry  
\begin{equation}
\label{eqn:symmetry}
C[2b_g,2c_g,\ldots,2b_1, 2c_1] \cong C[-2c_{1},-2b_1,\ldots,-2c_{g},-2b_g]
\end{equation}
to reduce the computation.

By definition, $C[2b_g,2c_g,\ldots,2b_1, 2c_1]$ is a positive diagram if and only if 
$b_i>0$ and $c_i<0$ for all $i$. Conversely, every positive 2-bridge knot has a positive diagram of this form.

\begin{proposition}
Every positive two-bridge knot $K$ of genus $g$ is represented by a positive diagram $K=C[2b_g,2c_g,\ldots,2b_1,2c_1]$ ($b_i>0, c_i<0$)
\end{proposition}
\begin{proof}
From the diagram $C[2b_g,2c_g,\ldots,2b_1,2c_1]$, we have a Seifert surface $S$ of genus $g$ that is the plumbing of the $2g$ annuli $A_{-c_1}, A_{b_1}, \ldots, A_{-c_g}, A_{b_g}$ (Figure \ref{fig:2-bridge} (ii)).

Let $\langle \; ,\; \rangle_S$ be the Seifert form of $S$. 
If $K$ is positive, the Seifert form of $S$ is positive definite because $\sigma(\langle\;, \; \rangle_S) = \sigma(K)=2g(K)$.
If $b_i<0$ or $c_i>0$ for some $i$, then the core curve $\gamma \subset S$ of the annulus $A_{b_i}$ or $A_{-c_i}$ satisfies $\langle \gamma,\gamma \rangle_S <0$. Therefore the diagram $C[2b_g,2c_g,\ldots,2b_1,2c_1]$ cannot represent a positive knot. Since every 2-bridge knot of genus $g$ is represented by a diagram $C[2b_g,2c_g,\ldots,2b_1,2c_1]$, this shows that a positive two-bridge knot of genus $g$ is always represented by a diagram $C[2b_g,2c_g,\ldots,2b_1,2c_1]$ with $b_i>0$ and $c_i<0$ for all $i$. 
\end{proof}

For $j=1,\ldots,g$, we define the \emph{$j$-th truncate of $K$} by
\begin{equation}
\label{eqn:truncate} K_j = C[2b_{j},2c_j,\ldots, 2b_1,2c_1]. 
\end{equation}
We define $d_j,p_j>0$ by the continued fraction
\begin{equation}
\label{eqn:continued-fraction}
\frac{d_j}{p_j} = 2b_j+\cfrac{1}{2c_j+\cfrac{1}{2b_{j-1} + \cfrac{1}{\ddots +\cfrac{1}{2c_1}} }} 
\end{equation}
By formally putting $d_0=1$ and $p_0=0$, for $i=0,\ldots,g-1$, $d_i$ and $p_i$ satisfies the following recurrence relations.
\begin{equation}
\label{eqn:recurrence}
\begin{cases}
d_{i+1} = (-4b_{i+1}c_{i+1}-1)d_{i} - 2b_{i+1}p_i \\
p_{i+1} = -2c_{i+1}d_i - p_i
\end{cases}
\end{equation}
In particular, it follows that $d_{i}>p_{i}>0$ and that $d_i$ and $p_i$ are coprime.
It is known that $d_{j}=\det(K_j)$. In particular, $\det(K)=\det(K_g)=d_g$.

By applying the Gauss diagram formula of $a_2$ \cite[Theorem 1]{pv} for the diagram $C[2b_g,2c_g,\ldots,2b_1,2c_1]$ we immediately get the following (c.f. \cite[Propositino 4.1]{iw}).

\begin{proposition}\label{prop:a2-formula}
If $K=C[2b_g,2c_g,\ldots,2b_1,2c_1]$ then
\[ a_2(K)= \sum_{i=1}^{g}\sum_{j=1}^{i}-b_ic_j\]
\end{proposition}

Similarly, by applying the Gauss diagram formula of $4v_3(K)$ \cite[Theorem 2]{pv}
we get the following (c.f. \cite[Proposition 4.4]{iw}).

\begin{proposition}\label{prop:v3-formula}
If $K=C[2b_g,2c_g,\ldots,2b_1,2c_1]$ then
\[ 4v_3(K)= \frac{1}{2}\left( \sum_{i=1}^{g}\sum_{j=1}^{i}-b_i^{2}c_j + \sum_{i=1}^{g}\sum_{j=1}^{i}b_ic_j^{2} \right) \]
\end{proposition}

For later use, we observe the following inequality.
\begin{lemma}
\label{lemma:1st-estimate}
Let $K=C[2b_g,2c_g,\ldots, 2b_1,2c_1]$ be a positive $2$-bridge knot.
Then 
\[ \det(K) \geq \left(\sum_{i=1}^{g}-2b_ic_i\right) +1 \] 
\end{lemma}
\begin{proof}
The proof is induction on $g$. For $g=1$, $\det(K)=-4b_1c_1-1 \geq -2b_1c_1 +1$.
For general case, since $d_g \geq p_g+1$, 
\begin{align*}
d_g-p_g &=(-4b_{g}c_g +2c_{g} -1)d_{g-1} + (-2b_g+1)p_{g-1} \\
&\geq (-4b_gc_g -2b_g +2c_g)p_{g-1} +(-4b_{g}c_g +2c_{g} -1)\\
&\geq (-4b_{g}c_g +2c_{g} -1) \\
&\geq -c_g
\end{align*}
Thus by induction
\begin{align*}
\det(K) &= (-4b_{g}c_g-1)d_{g-1} - 2b_gp_{g-1}\\
&= (-4b_{g}c_g-2b_g-1)d_{g-1} + 2b_g(d_{g-1}-p_{g-1})\\
& \geq d_{g-1} + (-2b_gc_g) = \left(\sum_{i=1}^{g}-2b_ic_i\right) +1
\end{align*}
\end{proof}

\subsection{Signature of torus knots}

For the $(2,2g+1)$-torus knot $T_{2,2g+1}$, its total $p$-signature $\sigma(T_{2,2g+1},p)$ is equal to the (ordinary) signature of the $(2g+1,p)$ torus link because they are equal to the signature of the $4$-manifold $X_{2,2g+1,p} = \{(z_1,z_2,z_3 \in \C^{3})\: | \: z_1^{2}+z_2^{2g+1}+z_3^{p} = \varepsilon, |z_1|^2 + |z_2|^2+|z_{3}|^2 \leq 1\}$ bounding the Brieskorn homology 3-sphere (here $\varepsilon>0$ is sufficiently small real number).
From the recurrence formula of the signature of torus knot \cite[Theorem 5.2]{glm}, we have the following estimate of the signature. 

\begin{lemma}
\label{lemma:signature-torus-knot}
Let
\[ a(q) = \begin{cases} 1 & q \mbox{:odd} \\ 2 & q \mbox{:even} \end{cases} \mbox{ and } \quad  b(q,n) = \begin{cases} \frac{1}{2} & q \mbox{:odd and } n \mbox{:even}  \\ 0 & \mbox{otherwise} \end{cases}\]
For the $(q,n)$ torus link $T_{q,n}$ with $1 \leq q \leq n$, 
\begin{equation}
\label{eqn:sgn}
\frac{(q-1)n}{2} \leq \sigma(T_{q,n}) \leq \frac{q(n+1)}{2} - a(q)-b(q,n)
\end{equation}
holds.
\end{lemma}

\begin{proof}
We prove the assertion by induction on $(q,n)$.
Assume that we have proven \eqref{eqn:sgn} for all $(q',n')$ with $(q',n')<(q,n)$ where $<$ is the lexicographical ordering. The case $q=1$ is obvious so in the following we assume that $q\geq 2$.\\

\noindent
{\bf Case 1. $q=n$}\\

By \cite[Theorem 5.2 (III)]{glm} $\sigma(T_{q,q}) =\frac{1}{2}( q^{2}-a(q))$ hence
\[ \frac{(q-1)q}{2} \leq \sigma(T_{q,q})= \frac{1}{2}( q^{2}-a(q)) \leq \frac{q(q+1)}{2} - a(q) -b(q,q) \]

\noindent
{\bf Case 2. $q < n < 2q$}\\

By \cite[Theorem 5.2 (III)]{glm}
\[ \sigma(T_{q,n}) = q^{2}-a(q) -\sigma(T_{q,2q-n})\]
Since $2q-n < q$ and $T_{q,2q-n}=T_{2q-n,q}$, by induction
\[ \frac{(2q-n-1)q}{2} \leq \sigma(T_{2q-n,q}) \leq \frac{(2q-n)(q+1)}{2}-a(2q-n)-b(2q-n,q) \]

First we prove the upper bound.
\begin{align*}
\sigma(T_{q,n}) &\leq q^{2}-a(q)-\frac{(2q-n-1)q}{2}=  \frac{q(n+1)}{2} - a(q)\\
\end{align*}
When $b(q,n)=0$, we have done. If $b(q,n)=\frac{1}{2}$, namely, $q$ is odd and $n$ is even, since $\sigma(T_{q,n})$ is an integer but $\frac{q(n+1)}{2}$ is not an integer,
\begin{align*}
\sigma(T_{q,n}) \leq \frac{q(n+1)}{2} - a(q) - \frac{1}{2} = \frac{q(n+1)}{2} - a(q)-b(q,n)
\end{align*}
holds.

For the lower bound, since $n-q \geq 1$
\begin{align*}
\sigma(T_{n,q}) & \geq q^{2}-a(q)-\frac{(2q-n)(q+1)}{2}+a(2q-n)+b(2q-n,q) \\
&= \frac{nq-2q+n}{2} -a(q)+a(2q-n)+b(2q-n,q) \\
&\geq \frac{(q-1)n}{2} + (n-q)-a(q)+a(2q-n)+b(2q-n,q)\\
&\geq \frac{(q-1)n}{2} 
\end{align*}

\noindent
{\bf Case 3. $n=2q$}\\

By \cite[Theorem 5.2 (II)]{glm} $\sigma(T_{q,2q})=q^{2}-1$ hence
\[ \frac{(q-1)(2q)}{2} \leq \sigma(T_{q,2q})=q^{2}-1 \leq \frac{q(2q+1)}{2} - a(q)-b(q,2q)\]

\noindent
{\bf Case 4. $2q < n < 3q$}\\

By \cite[Theorem 5.2 (I)]{glm}
\[ \sigma(T_{q,n}) =\sigma(T_{q,n-2q}) + q^{2} +a(q)-2 \]
Since $T_{q,n-2q}=T_{n-2q,q}$ and $0 < n-2q < q$, by induction
\[  \frac{(n-2q-1)q}{2} \leq \sigma(T_{n-2q,q}) \leq \frac{(n-2q)(q+1)}{2} -a(n-2q)-b(n-2q,q) \]
Therefore
\begin{align*}
\sigma(T_{q,n}) 
&\leq \frac{(n-2q)(q+1)}{2} - a(n-2q)-b(n-2q,q)+ q^{2} + a(q) -2 \\
&\leq \frac{nq+n-2q}{2}-a(n-2q)-b(n-2q,q)+a(q)-2
\end{align*}
If $n\leq 3q-2$, $n-2q \leq q-2$ so
\begin{align*}
\frac{nq+n-2q}{2}-a(n-2q)+a(q)-2 & \leq \frac{q(n+1)}{2}-1-a(n-2q)-b(n-2q,q)+a(q)-2 \\
& \leq \frac{q(n+1)}{2}-a(q)
\end{align*}
If $n=3q-1$, $n \not \equiv q \pmod 2$. Thus
\[ \frac{nq+n-2q}{2}-a(n-2q)-b(n-2q,q)+a(q)-2 = \frac{q(n+1)}{2}-\frac{1}{2}-2-0+1-2 \leq \frac{q(n+1)}{2} -1 \]
when $q$ is odd, and 
\[ \frac{nq+n-2q}{2}-a(n-2q)-b(n-2q,q)+a(q)-2 = \frac{q(n+1)}{2}-\frac{1}{2}-1-\frac{1}{2}+2-2 \leq \frac{q(n+1)}{2} -2 \]
when $q$ is even.

For the lower bounds,
\begin{align*}
\sigma(T_{q,n})
&\geq \frac{(n-2q-1)q}{2}+q^{2} +a(q)-2 \\
&= \frac{(n-1)q}{2} + a(q) -2 \\
&\geq \frac{(n-1)q}{2}
\end{align*}

\noindent
{\bf Case 5. $n \geq 3q$}\\

By \cite[Theorem 5.2 (I)]{glm}
\[ \sigma(T_{q,n}) =\sigma(T_{q,n-2q}) + q^{2} +a(q)-2. \]
Since $q \leq n-2q$, by induction
\[  \frac{(q-1)(n-2q)}{2} \leq \sigma(T_{q,n-2q}) \leq \frac{q(n-2q+1)}{2} - a(q) - b(q,n-2q) \]
\begin{align*}
\sigma(T_{n,q}) 
&\leq \frac{q(n-2q+1)}{2} - a(q) -b(q,n-2q)+ q^{2} + a(q) -2 \\
&\leq \frac{q(n+1)}{2}-b(q,n)-2 \\
&\leq \frac{q(n+1)}{2}-a(q)-b(q,n)
\end{align*}
and
\begin{align*}
\sigma(T_{n,q}) 
&\geq \frac{(q-1)(n-2q)}{2}+q^{2} +a(q)-2 \\
&= \frac{(q-1)n}{2} +q + a(q) -2 \\
&\geq \frac{(q-1)n}{2}
\end{align*}

\end{proof}

\subsection{Main obstruction}

Combining Theorem \ref{theorem:main-obst} and Lemma \ref{lemma:signature-torus-knot} we get the following main obstruction for a positive two-bridge knot to admit no chirally cosmetic surgeries.

\begin{corollary}
\label{cor:main-obstruction}
Let $K$ be a  positive $2$-bridge knot other than the $(2,k)$ torus knot.
 If 
\begin{equation}
\label{eqn:main-obstruction} \det(K) + 6g(K)-5 >
\begin{cases}
 \displaystyle \frac{176a_2(K)}{10g(K)} & (g\geq 6) \\
\displaystyle \frac{16a_2(K)}{g(K)} & (g\leq 5) \\
\end{cases}
\end{equation}
Then $K$ has no chirally cosmetic surgeries.
\end{corollary}
In the following we call the inequality \eqref{eqn:main-obstruction} the \emph{main obstruction}.

\begin{proof}
For each $\omega \in \{z \in \C \: | \: |z|=1\}$, if a knot $K$ is obtained from $K'$ by positive-to-negative crossing changes, then $\sigma_{\omega}(K) \leq \sigma_{\omega}(K')$ holds. From the Conway's normal form we see that every positive two bridge knot $K$ of genus $g$ can be converted to the $(2,2g+1)$ torus knot $T_{2,2g+1}$ by the positive-to-negative crossing changes. Thus
\begin{equation}
\label{eqn:sig-bound} \phi^{\geq N}(T_{2,2g+1}) \leq  \phi^{\geq N}(K) 
\end{equation}
holds for every $N\geq 0$.

By Lemma \ref{lemma:signature-torus-knot} for $p \geq 11$,
\[ \frac{\sigma(T_{2,2g+1},p)}{pg} \geq \frac{(2g+1)(p-1)}{2pg} \geq \frac{10}{11} \]
if $2g+1 > p$ and 
\[ \frac{\sigma(T_{2,2g+1},p)}{pg} \geq \frac{2gp}{2pg}= 1 \]
if $2g+1 \leq p$.
When $g\leq 5$, since $p \geq 11 \geq 2g+1$, the former case does not happen.
Therefore
\[ \phi^{\geq 11}(T_{2,2g+1}) \geq
\begin{cases}
\frac{10}{11} & g\geq 6 \\
1 & g\leq 5
\end{cases}
\]
Thus Theorem \ref{theorem:main-obst} and \eqref{eqn:sig-bound} give the desired constraint.
\end{proof}

The proof actually tells us that the main obstruction \eqref{eqn:main-obstruction} is valid for a special alternating knot $K$ of genus $g$, if $K$ be converted to the $(2,2g(K)+1)$-torus knot by the positive-to-negative crossing changes.

More generally, we have the following variant of Theorem \ref{theorem:main-obst} that deserves to state. 
\begin{theorem}
\label{theorem:main-variant}
Assume that $K$ satisfies the following properties.
\begin{itemize}
\item[(i)] $a_2(K)>1$.
\item[(v)] $K$ is Heegaard Floer homologically thin.
\item[(vi)] $K$ is not an L-space knot.
\item[(vii)] $K$ can be converted to the $(2,2g(K)+1)$-torus knot by the positive-to-negative crossing changes.
\end{itemize}
If $K$ satisfies the inequality
\[ \det(K) + 6g(K)-5 >
\begin{cases}
 \displaystyle \frac{176a_2(K)}{10g(K)} & (g(K)\geq 6) \\
\displaystyle \frac{16a_2(K)}{g(K)} & (g(K)\leq 5) \\
\end{cases} \]
then $K$ has no chirally cosmetic surgeries.
\end{theorem}
\begin{proof}
Assume that $S^{3}_K(\frac{p}{q}) \cong -S^{3}_K(\frac{p}{q'})$ for $q \neq q'$. By the assumption (vi) we have $qq'<0$. The property (iii) of Theorem \ref{theorem:main-obst} follows from (vii) because $\sigma(K,p) \geq \sigma(T_{2,2g(K)+1},p)$. Thus by the constraint from Casson-Walker and the Casson-Gordon invariant \eqref{eqn:CG-CW-obst} $q+q' \neq 0$.

If a knot $K'$ is obtained from a knot $K$ by positive-to-negative crossing change, $\tau(K')\leq \tau(K)$ \cite[Corollary 1.5]{os-tau}, \cite[Corollary 3]{li}. Thus the property (iv) of Theorem \ref{theorem:main-obst} follows from the assumption (vii), because $g(K)= \tau(T_{2,2g(K)+1}) \leq \tau(K)\leq g(K)$. Therefore we can apply Heegaard Floer homology obstruction \eqref{eqn:Heegaard}. Thus by \eqref{eqn:Heegaard} and \eqref{eqn:CG-CW-obst},
\[ \frac{4a_2(K)p}{\sigma(K,p)} =  \frac{1}{2} \left(\frac{1}{2}\det(K)+ 3g(K)-\frac{5}{2} \right) \]
The rest of the arguments is the same as that of Corollary \ref{cor:main-obstruction}.
\end{proof}

\section{Non-existence of chirally cosmetic surgeries}
\label{section:proof}

In this section we prove Theorem \ref{theorem:main}.

We already have a formula of $a_2(K)$ (Proposition \ref{prop:a2-formula}), $4v_3(K)$ (Proposition \ref{prop:v3-formula}) and $\det(K)$ \eqref{eqn:continued-fraction}. Although we did not do this, it is also feasible to get a formula of $a_4(K)$.
Nevertheless, it is not easy to use them to apply Theorem \ref{theorem:main-obst} because there are many variables and the formula will be complicated.

Our proof of Theorem \ref{theorem:main} is divided into two steps.
\begin{description}
\item[Step A] We show that most positive two bridge knot satisfy the main criterion \eqref{eqn:main-obstruction}.
\item[Step B] We show the remaining cases has no chirally cosmetic surgeries by using equality part of Theorem \ref{theorem:main-obst}. Namely, we show that they violates the equality
\begin{equation}
\label{eqn:sub-obst}
\frac{2(7a_2(K)^2-a_2(K)-10a_4(K))}{4v_3(K)} = \det(K) + 6g(K)-5
\end{equation}
\end{description}

Since the knots which we need to treat in Step B is quite limited, it becomes feasible to check \eqref{eqn:sub-obst} by direct computations.

Our strategy for step A is to use induction argument.
For a positive $2$-bridge knot $K=C[2b_g,2c_g,\ldots,2b_1,2c_1]$ of genus $g$, we put
\[ K^{b_i+} = C[2b_g,2c_g,\ldots,2b_{i+1}, 2c_{i+1}, 2(b_i+1), 2c_i, \ldots,2b_1,2c_1] \]
and
\[ K^{c_i-} =  C[2b_g,2c_g,\ldots,2b_{i+1}, 2c_{i+1}, 2b_i, 2(c_i-1), \ldots,2b_1,2c_1] \]
We show that, once the main obstruction \eqref{eqn:main-obstruction} is satisfied for $K$, then $K^{b_i+}$ and $K^{c_i-}$ also satisfy the main obstruction \eqref{eqn:main-obstruction} in most cases. 

\subsection{Induction argument condition}
In this section we give a sufficient condition for the induction argument to work.

\begin{proposition}
\label{prop:induction-criteria}
Let $K$ be a positive $2$-bridge knot of genus $g$. Assume that $K' = K^{b_i+}$ or $K^{c_i-}$. If $K$ satisfies \eqref{eqn:main-obstruction} and
\begin{equation} 
\label{eqn:condition}
\frac{\det(K')-\det(K)}{a_2(K') - a_2(K)} \geq 
\begin{cases} 
\frac{176}{10g} & (g \geq 6) \\
\frac{16}{g} & (g \leq 5) \\
\end{cases} 
\end{equation}
then $K'$ also satisfies \eqref{eqn:main-obstruction}.
\end{proposition}
\begin{proof}
Let $\alpha =\frac{\det(K')-\det(K)}{a_2(K') - a_2(K)}$. Since $K$ satisfies  \eqref{eqn:main-obstruction},
\begin{align*}
\det(K')+6g-5 &= \det(K) + 6g-5 + \alpha(a_2(K') - a_2(K)) \\
& \geq \begin{cases}
\frac{176a_2(K')}{10g} + (a_2(K') - a_2(K)) \left(\frac{-176}{10g} + \alpha \right) & (g\geq 6) \\
\frac{16a_2(K')}{g} +(a_2(K') - a_2(K)) \left(\frac{-16}{g} + \alpha \right) & (g \leq 5)
\end{cases}
\end{align*}
Thus if $\alpha$ satisfies \eqref{eqn:condition}, $K'$ also satisfies \eqref{eqn:main-obstruction}.
\end{proof}

The next lemma evaluates $\frac{\det(K')-\det(K)}{a_2(K') - a_2(K)}$.

\begin{lemma}
\label{lemma:diff-estimate}
Let $K=C[2b_g,2c_g,\ldots,2b_1,2c_1]$ be a positive $2$-bridge knot of genus $g$.
\begin{itemize}
\item[(i)] $\det(K^{b_i+}) - \det(K) \geq (8(g-i)+4)(a_2(K^{b_i+})-a_2(K))$.
\item[(ii)] $\det(K^{c_i-}) - \det(K) \geq (8i-4)(a_2(K^{c_i-})-a_2(K))$.
\end{itemize}
\end{lemma}
\begin{proof}
(i) By Proposition \ref{prop:a2-formula},
\[ (a_2(K^{b_i+})-a_2(K))= -c_1-c_2-\cdots-c_{i}. \]
For $j=1,\ldots, m$, we put $x_j=d(K^{b_i+}_j)- d(K_j)$ and $y_j =p(K^{b_i+}_j)- p(K_j)$. Here $K^{b_i+}_j$ and $K_j$ are the $j$-th truncate of $K^{b_i+}$ and $K$, defined in \eqref{eqn:truncate}.

It is clear that $x_1=\cdots = x_{i-1}=0$ and $y_1 = \cdots =y_{i-1}= 0$.
We prove the following claim that implies the deisred assertion.
\begin{claim}
For $j\geq i$,
\begin{equation*}
 x_j-y_j \geq  4(-c_1-c_2-\cdots-c_{i})
\end{equation*}
and
\begin{equation*}
x_j \geq (8(j-i)+4)(-c_1-c_2-\cdots-c_{i})
\end{equation*}
hold. 
\end{claim}
\begin{proof}[Proof of Claim]
We prove the claim by induction on $j$. For $j=i$, $y_i=0$. By the recurrence relation \eqref{eqn:recurrence} of $d_i$ and $p_i$,
\begin{align*}
x_{i} 
&= -4c_{i}d_{i-1}-2p_{i-1}\\
&= (-4c_{i}-2)d_{i-1} + 2(d_{i-1}-p_{i-1}) \\
&\geq ( -4c_i-2)(-2c_1-2c_2-\cdots -2c_{i-1} +1) + 2 \\ 
& \geq 4(-c_1 - c_2 - \cdots - c_{i})
\end{align*}
Here at the third inequality we used Lemma \ref{lemma:1st-estimate} that says that 
\[ d_{i-1} =\det(K_{i-1}) \geq \left( \sum_{k=1}^{i-1}-2b_{k}c_{k}\right) + 1 \geq  (-2c_1-2c_2-\cdots -2c_{i-1} +1) \]
and the fact that $d_{i-1}-p_{i-1} \geq 1$.

For $j>i$, by the linear recurrence relations \eqref{eqn:recurrence}, it follows that $x_i$ and $y_i$ also satisfy the same recurrence relation
\begin{equation}
\label{eqn:rec-x-y}
\begin{cases}
x_{j} = (-4b_jc_j-1)x_{j-1} -2b_j y_{j-1}\\
y_{j}=-2c_jx_{j-1} -  y_{j-1}
\end{cases}
\end{equation}
Thus
\begin{align*}
x_j-y_j
&= (-4b_jc_j+2c_j-1) x_{j-1} - (2b_j-1) y_{j-1} \\
&= (-4b_jc_j-2b_j+2c_j)x_{j-1} + (2b_j-1)(x_{j-1}- y_{j-1}) \\
& \geq (2b_{j}-1)(x_{j-1}-y_{j-1}) \\
& \geq 4(-c_1-c_2-\cdots-c_{i})
\end{align*}
and
\begin{align*}
x_j
&= (-4b_jc_j-1)x_{j-1} -2b_j y_{j-1} \\
&= (-4b_jc_j-2b_j -1) x_{j-1} + 2b_{j}(x_{j-1}-y_{j-1}) \\
&\geq x_{j-1} +8(-c_1 - c_2 - \cdots - c_{i})\\
&\geq \left( (8(j-1-i)+4) + 8 \right) (-c_1 - c_2 - \cdots - c_{i})\\
& = \left( 8(j-i)+4 \right) (-c_1 - c_2 - \cdots - c_{i})
\end{align*}
\end{proof}

(ii)
By Proposition \ref{prop:a2-formula}
\[ (a_2(K^{c_i-})-a_2(K))= b_i+b_{i+1}+\cdots + b_{g} \]

For $j=1,\ldots, m$, we put $x_j=d(K^{c_i-}_j)- d(K_j)$ and $y_j =p(K^{c_i-}_j)- p(K_j)$.  It is clear that $x_1=\cdots = x_{i-1}=0$ and $y_1 = \cdots =y_{i-1}= 0$.
As in the proof of (i), we prove the following claim that implies the desired assertion.

\begin{claim}
For $j\geq i$
\begin{equation*}
\label{eqn:c_x-y} 
x_j-y_j \geq 4i-2
\end{equation*}
and
\begin{equation*}
\label{eqn:c_x} x_j \geq (8i-4)(b_i+b_{i+1}+\cdots +b_{j})
\end{equation*}
holds.
\end{claim}
\begin{proof}[Proof of Claim]
We prove the claim by induction on $j$.

For $j=i$, $x_{i} = 4b_{i}d_{i-1}, \quad y_i=2p_{i-1}$.
By Lemma \ref{lemma:1st-estimate} $d_{i-1} \geq 2i-1$ hence
\[x_i \geq 4b_i(2i-1) \geq (8i-4)b_i \]
and
\[ x_{i}-y_{i} = (4b_{i}-2)d_{i-1} + 2(d_{i-1}-p_{i-1}) \geq 2b_{i}d_{i-1} + 2 > 4i-2 \]
holds.

For $j>i$, since $x_j$ and $y_j$ satisfies the same recurrence relations \eqref{eqn:rec-x-y} so by the same argument as (i), we get
\[ (x_j-y_j) \geq (2b_j-1)(x_{j-1}-y_{j-1}) \geq 4i-2.\]
Finally, 
\begin{align*}
x_j
&= (-4b_jc_j-1)x_{j-1} -2b_j y_{j-1} \\
&= (-4b_jc_j-2b_j -1) x_{j-1} + 2b_{j}(x_{j-1}-y_{j-1}) \\
& \geq x_{j-1} + 2b_j(4i-2) \\
& \geq (8i-4)(b_{i}+\cdots + b_{j-1} + b_{j})
\end{align*}
\end{proof}
\end{proof}

By Proposition \ref{prop:induction-criteria} and Lemma \ref{lemma:diff-estimate}, we have the following conditions for induction argument to work.
\begin{corollary}
\label{cor:induction}
Let $K$ be a positive 2-bridge knot of genus $g$.
\begin{itemize}
\item[(i)] If $g \geq 4$ and $K$ satisfies \eqref{eqn:main-obstruction}, then $K^{b_i+}$ and $K^{c_i-}$ also satisfy \eqref{eqn:main-obstruction}.
\item[(ii)] If $2 \leq g \leq 3$ and $K$ satisfies \eqref{eqn:main-obstruction}, then then $K^{b_i+}$ and $K^{c_i-}$ also satisfy \eqref{eqn:main-obstruction}, except $K^{b_g+}$ and $K^{c_1-}$.
\end{itemize}
\end{corollary}

In the remaining sections, we prove Theorem \ref{theorem:main}, the non-existence of chirally cosmetic surgeries. 

\subsection{Genus $\geq 4$ cases}

For $g \geq 4$ case, by Corollary \ref{cor:induction} (i), to check \eqref{eqn:main-obstruction} it is sufficient to treat positive two bridge knots with small $b_i$ and $-c_i$.

Let $\delta(K) = \left(\sum_{i=g}b_i + \sum_{i=1}^{g}(-c_i)\right) -2g$. 
$\delta(K)=0$ is and only if $K$ is the $(2,2g+1)$-torus knot. It is easy to see that $(2,2g+1)$ torus knot does not satisfy the main obstruction \eqref{eqn:main-obstruction}. However for $\delta(K)=1$ case we have the following.

\begin{lemma}
If $g\geq 4$ and $\delta(K)=1$ then $K$ satisfies \eqref{eqn:main-obstruction} unless $b_g=2$ or $c_1=-2$.
\end{lemma}
\begin{proof}
By the symmetry \eqref{eqn:symmetry}, it is sufficient to show the case $b_i=2$ for some $i=1,\ldots,g-1$. 
In this case,  $a_2(K) - a_2(T_{2,2g+1}) = i \leq g$. 
By Lemma \ref{lemma:diff-estimate}
\[ \det(K)-\det(T_{2,2g+1}) \geq (8(g-i)+4)i \geq 7g \]
Therefore
\begin{align*}
\det(K)+6g-5 & \geq \det(T_{2,2g+1}) + 7g + 6g-5 =  15g-4
\end{align*}
On the other hand, since $a_2(K) < a_2(T_{2,2g+1})+g = \frac{g(g+1)}{2}$ 
\[ \frac{176a_2(K)}{10g} < \frac{176a_2(T_{2,2g+1})}{10g} + \frac{176}{10} = \frac{176}{20}(g+1) + \frac{176}{10} \leq 15g-4 \]
if $g\geq 6$ and
\[ \frac{16a_2(K)}{g} < \frac{16a_2(T_{2,2g+1})}{g} + 16 = 8(g+1) + 16 \leq 15g-4\]
if $g=4,5$.
Hence $K$ satisfies \eqref{eqn:main-obstruction}.
\end{proof}

Thus if a positive $2$-bridge knot $K$ has chirally cosmetic surgeries, then $K$ is  of the form
\[ K=[2b_g,-2,2,-2,\ldots,2,-2c_1] \]
For such a knot $K$,
\begin{align*}
\det(K)&= -8b_gc_1 g + 4b_1c_1-4b_gg + 4c_1g + 4b_g -4c_g +2g-3\\
a_2(K) &= -b_gc_1 +b_gg - c_1g -b_g + c_1 + \frac{g^{2}-3g+2}{2} 
\end{align*}
holds. By direct computation, one can check that \eqref{eqn:main-obstruction} holds if 
\begin{itemize}
\item $(b_g,c_1)=(2,-2)$ and $g\geq 4$
\item $(b_g,c_1)=(1,-4)$ and $g\geq 5$
\item $(b_g,c_1)=(4,-1)$ and $g\geq 5$
\item $(b_g,c_1)=(1,-3)$ and $g\geq 8$
\item $(b_g,c_1)=(3,-1)$ and $g\geq 8$
\end{itemize}
By Corollary \ref{cor:induction} (i) this implies that $K$ satisfies the main obstruction \eqref{eqn:main-obstruction} unless

\begin{itemize}
\item[(a)] $g=4$ and $K=[2b_g,-2,2,-2,\ldots,2,-2]$ $(b_g \geq 1)$ or $K=[2,-2,\ldots,2c_1]$ ($c_1 \leq -1$)
\item[(b)] $g=5,6,7$ and $K=[2b_g,-2,2,-2,\ldots,2,-2]$ $(b_g=1,2,3)$ or $K=[2,-2,\ldots,2c_1]$ ($c_1=-1,-2,-3$)
\item[(c)] $g\geq 8$ and $K=[2b_g,-2,2,-2,\ldots,2,-2]$ $(b_1=1,2)$ or $K=[2,-2,\ldots,2c_1]$  ($c_1=-1,-2$).
\end{itemize}

Now we move to step B. We show that the knots in (a), (b), (c) do not satisfy \eqref{eqn:sub-obst}.

By the symmetry \eqref{eqn:symmetry} $K=[2b_g,-2,2,-2,\ldots,2,-2] = [2,-2,2,-2,\ldots,2,-2b_g]$ so it sufficient to treat the case $K=[2b_g,-2,2,-2,\ldots,2,-2]$ and $b_g>1$.
For such a knot $K$, by direct computations
\begin{align*} a_2(K)&= \frac{(2b_g+g-1)g}{2} \\
a_4(K)&=\frac{(4b_g+g-2)g(g^{2}-1)}{24} \\
4v_3(K)&= \frac{g(b_g^2+b_g+g-1)}{2}\\
\det(K)&= 4b_gg-2g+1
\end{align*}
From these computations, one can directly check that $K$ does not satisfy \eqref{eqn:sub-obst} for the cases (a), (b), (c).
%

This completes the proof of non-existence of chirally cosmetic surgeries for positive 2-bridge knots other than the $(2,2g+1)$-torus knots for $g\geq 4$ case.

\subsection{Genus $\leq 3$ cases}

For $g \leq 3$ case, by Corollary \ref{cor:induction} (ii) we need some additional care for the treatment of the parameters $b_1$ and $c_g$. We treat the cases $g=1,2,3$ separately.

\subsubsection{Genus one case}
It has already shown that chirally comsetic surgeries of genus one alternating knot (that of course includes the genus one positive 2-bridge knot) is either (A) or (B) in introduction \cite[Theorem 6.4]{iis1}.
Indeed, one can directly show that genus one positive two-bridge knot $K$ does not satisfy \eqref{eqn:sub-obst}.

\subsubsection{Genus two case}

To avoid subscript and minus signs, we put 
\[ K=C[2x,-2y,2z,-2w] \qquad (x,y,z,w>0) \]
Since
\begin{align*}
\det(K) &=16xyzw-4yx-4zw-4xw+1\\
a_2(K) &= zw+xw+xy, 
\end{align*}
The main obstruction \eqref{eqn:main-obstruction} is satisfied if and only if
\begin{equation}
\label{eqn:g2-obst}
4xyzw-3yx-3zw-3xw +2 \geq 0
\end{equation}
By a routine calculation, \eqref{eqn:g2-obst} is satisfied if 
\begin{itemize}
\item[(a)] $y,z\geq 2$. 
\item[(b)] $y=2$, $z=1$, and $w \geq 2$.	
\item[(c)] $y=1$, $z=2$, and $x \geq 2$.
\end{itemize}
Thus thanks to Corollary \ref{cor:induction} (ii), positive 2-bridge knots of genus two satisfies the main obstruction \eqref{eqn:main-obstruction} possibly except
\begin{itemize}
\item[(a)] $C[2x,-2,2,-2w]$,
\item[(b)] $C[2x,-4,2,-2]$.
\item[(c)] $C[2,-2,4,-2w]$.
\end{itemize}

Now we move to Step B. We show that the knots in (a), (b), (c) do not satisfy \eqref{eqn:sub-obst}.

For the case (a), by the symmetry \eqref{eqn:symmetry}, we may assume that $x\geq w$. The invariants appeared in our constraints is calculated as follows.
\begin{align*}
a_2(K) & =xw+x+w\\
a_4(K) & =xw\\
\det(K) &=12xw-4x-4w+1\\
4v_3(K) &=\frac{1}{2}(x^2 w + w^2 x + x^2 +w^2 +x+w)
\end{align*}

From the formula, it follows that if $w=1$ then $K$ does not satisfy \eqref{eqn:sub-obst}. 
If $w \geq 2$ then $x\geq 2$ and
\[ \det(K)+6g-5=12xw-4x-4w+8 \geq 6xw+6x+6w = 6a_2(K)\]
Furthermore, in this case
\[ \frac{x+w}{2}a_2(K) \leq 4v_3(K)\]
hence
\[ \frac{2(7a_2(K)^2-a_2(K)-10a_4(K))}{4v_3(K)} \leq \frac{28}{x+w}a_2(K) \]
Thus if $K$ satisfies \eqref{eqn:sub-obst}
\[ 6a_2(K) \leq \det(K)+6g-5 = \frac{2(7a_2(K)^2-a_2(K)-10a_4(K))}{4v_3(K)} \leq \frac{28}{x+w}a_2(K) \]
so $x+w\leq 4$. This shows that $(x,w)=(2,2)$. However, $C[4,-2,2,-4]$ does not satisfy \eqref{eqn:sub-obst}.

To see case (b) and (c), by the symmetry \eqref{eqn:symmetry} $C[2x,-4,2,-2] = C[2,-2,4,-2x]$ so it is sufficient to treat the case (b).
By direct computation one can check that $C[2x,-4,2,-2]$ does not satisfy \eqref{eqn:sub-obst}.

\subsubsection{Genus three case}

As in the genus two case to avoid subscripts and signs, we put 
\[ K=C[2x,-2y,2z,-2w,2u,-2v] \qquad (x,y,z,w,u,v>0). \]

\begin{lemma}
\label{lemma:genus-3}
$K$ has no chirally cosmetic surgeries unless $y=z=w=u=1$.
\end{lemma}
\begin{proof}
Let $K=C[2x,-2,4,-2,2,-2v]$. For this knot
\begin{align*}
\det(K) &= 68xv-24x-20v+7\\
a_2(K) &= xv+2x+3v+2 
\end{align*}
hence by direct computation, $K$ satisfies the main obstruction \eqref{eqn:main-obstruction} for all $x,v$. Hence by Corollary \ref{cor:induction} (ii), a positive two-bridge knot $C[2x,-2y,2z,-2w,2u,-2v]$ has no chirally cosmetic surgeries if $z \neq 1$. By the symmetry \eqref{eqn:symmetry}, it follows that  $C[2x,-2y,2z,-2w,2u,-2v]$ has no chirally cosmetic surgeries if $w \neq 1$.

Similarly, let $K=C[2x,-4,2,-2,2,-2v]$. For this knot
\begin{align*}
\det(K) &= 52xv-20x-8v+3\\
a_2(K)&=xv+2v+3x+1 
\end{align*}
hence by direct computation $K$ satisfies the main obstruction \eqref{eqn:main-obstruction} for all $x,v$. Thus by the same argument a positive two-bridge knot $C[2x,-2y,2z,-2w,2u,-2v]$ has no chirally cosmetic surgeries if $y \neq 1$ or $u \neq 1$.
\end{proof}

Thus it remains to treat the knot $K=C[2x,2,-2,2,-2,-2v]$.
Then
\begin{align*}
\det(K)&=20xv-8x-8v +3 \\
a_2(K)&=xv+2x+2v+1
\end{align*} 
By the symmetry \eqref{eqn:symmetry}, we may assume that $x\geq v$. 
By direct computation, $C[2x,2,2,2,2,-2v]$ satisfies the main obstruction \eqref{eqn:main-obstruction} if $x \geq 3$ and $v \geq 2$.

Finally, by direct computations, the remaining knots
$C[4,-2,2,-2,2,-4]$ and $C[2x,-2,2,-2,2,-2]$ $(x \geq 2)$ do not satisfy \eqref{eqn:sub-obst} hence they have no chirally cosmetic surgeries.

\medskip

This completes the proof of Theorem \ref{theorem:main}.


\begin{thebibliography}{1}
\bibitem[BZH]{bz}
G.\ Burde, H.\ Zieschang, and M.\ Heusener, Knots. extended edn., De Gruyter Stud. Math. 5, De Gruyter, Berlin, 2014.
%
\bibitem[CG]{cg}
T.\ Cochran and R.\ Gompf,
{\em Applications of Donaldson's theorems to classical knot concordance, homology 3-spheres and property P.} Topology 27 (1988), no.4, 495--512.


\bibitem[FPS]{fps} D. Futer, J. Purcell and S. Schleimer,
{\em Effective bilipschitz bounds on drilling and filling,}
Geom. Topol. 26 (2022), no. 3, 1077--1188.

\bibitem[GLM]{glm}
C.\ Gordon, R.\ Litherland, and K.\ Murasugi,
{\em Signatures of covering links.}
Canadian J. Math. 33(1981), no.2, 381--394.

\bibitem[HLTZ]{hltz}M.\ Huang, Z.\ Li, R.\ Tanaz, and C.\ Zhang,
{\em Positive 2-bridge knots and chirally cosmetic surgeries.}
arXiv:2308.10126



\bibitem[IIS1]{iis1} K.\ Ichihara, T.\ Ito and T.\ Saito,
{\em Chirally cosmetic surgeries and Casson invariants.}
Tokyo J. Math, 44 (2021), no. 1, 1--24.


\bibitem[IIS2]{iis2} K.\ Ichihara, T.\ Ito and T.\ Saito,
{\em On constraints for knots to admit chirally cosmetic surgeries and their calculations.}
Pacific J. Math. 321 (2022), no. 1, 167--191.

\bibitem[It1]{it-LMO} T.\ Ito,
{\em On LMO invariant constraints for cosmetic surgery and other surgery problems for knots in $S^3$.}
Comm. Anal. Geom.28(2020), no.2, 321--349.

\bibitem[It2]{it-cable}
T.\ Ito, 
{\em A note on chirally cosmetic surgery on cable knots.}
Canad. Math. Bull.64(2021), no.1, 163--173.

\bibitem[It3]{it-special} T.\ Ito,
{\em Special alternating knots with sufficiently many twist regions have no chirally cosmetic surgeries}
arXiv:2301.09855.

\bibitem[IW]{iw}
K.\ Ichihara,  and Z.\ Wu, 
{\em A note on Jones polynomial and cosmetic surgery.}
Comm. Anal. Geom.27(2019), no.5, 1087--1104.

\bibitem[Li]{li}
C.\ Livingston, 
{\em Computations of the Ozsv\'ath-Szab\'o knot concordance invariant.}
Geom. Topol.8(2004), 735--742.

\bibitem[OS1]{os-alt}
P.\ Ozsv\'ath and Z.\ Szab\'o,
{\em Heegaard Floer homology and alternating knots.} Geom. Topol. 7 (2003), 225--254.

\bibitem[OS2]{os-tau}
P.\ Ozsv\'ath and Z.\ Szab\'o,
{\em Knot Floer homology and the four-ball genus.} Geom. Topol. 7 (2003), 615--639.

\bibitem[OS3]{os-lens}
P.\ Ozsv\'ath and Z.\ Szab\'o,
{\em On knot Floer homology and lens space surgeries.} Topology 44 (2005), no.6, 1281--1300.

\bibitem[OS4]{os-rational}
P.\ Ozsv\'ath and Z.\ Szab\'o,
{\em Knot Floer homology and rational surgeries.}
Algebr. Geom. Topol.11(2011), no.1, 1--68.

\bibitem[PV]{pv}
M.\ Polyak and O.\ Viro,
{\em Gauss diagram formulas for Vassiliev invariants,}
Int. Math. Res. Not. 1994(11) (1994) 445--453. 

\bibitem[Pr]{pr}
J.\ Przytycki, 
{\em Positive knots have negative signature.}
Bull. Polish Acad. Sci. Math.37 (1989), no.7--12, 559--562.


\bibitem[Ro]{ro} Y. \ Rong, 
{\em Some knots not determined by their complements.}
 In: Quantum topology, Ser.
Knots Everything, 3,World Sci. Publ., River Edge, NJ, 1993, pp. 339--353.

\bibitem[St]{st}
A.\ Stoimenow, 
{\em Positive knots, closed braids and the Jones polynomial.}
Ann. Sc. Norm. Super. Pisa Cl. Sci. (5) 2 (2003), no. 2, 237--285.


\bibitem[St2]{st-determinant}
A. Stoimenow,
{\em Graphs, determinants of knots and hyperbolic volume.} Pacific J. Math. {\bf 232} (2007), no.~2, 423--451.

\bibitem[Va1]{va1} K. Varvarezos,
{\em Alternating odd pretzel knots and chirally cosmetic surgeries,}
J. Knot Theory Ramifications 31 (2022), no. 6, Paper No. 2250045, 20 pp.

\bibitem[Va2]{va} K.\ Varvarezos,
{\em Heegaard Floer homology and chirally cosmetic surgeries,}
arXiv:2112.03144v1.



\end{thebibliography}
\end{document}